\documentclass[11pt,reqno]{amsart}
\usepackage{amsthm,amsfonts,amssymb,euscript}

\def\ddb#1{\sqrt{-1}\partial\bar{\partial}#1}
\def\dt#1{\frac{\partial}{\partial t}#1}

\newcommand{\B}{{\mathbf B}_{\beta}}

\newcommand{\C}{{\mathbb C}^{n}}

\newcommand{\dr}{\omega}
\newcommand{\ka}{K\"{a}hler}

\theoremstyle{plain}
  \newtheorem{theorem}{Theorem}[section]
  \newtheorem{proposition}[theorem]{Proposition}
  \newtheorem{lemma}[theorem]{Lemma}
  \newtheorem{corollary}[theorem]{Corollary}

\theoremstyle{definition}

\numberwithin{equation}{section}
\begin{document}

\title[$C^{2,\alpha}$-estimate for conical K\"{a}hler-Ricci flow]{$C^{2,\alpha}$-estimate for conical K\"{a}hler-Ricci flow}
\author{Liangming Shen}
\address{Department of Mathematics, Princeton University, Princeton, NJ 08544, USA.}
\email{liangmin@math.princeton.edu}

\begin{abstract}
      In this note, we establish a parabolic version of Tian's $C^{2,\alpha}$-estimate for conical complex Monge-Ampere equations \cite{Ti14}, 
which includes conical \ka-Einstein metrics. Our estimate will complete the proof of the existence of
unnormalized conical \ka-Ricci flow in \cite{Sh}.
\end{abstract}

\maketitle

%\tableofcontents

\section{introduction}\label{section1}

Regularity of complex Monge-Ampere equations in conic setting is an important issue in the study of conic \ka\ metrics. There have been
some works on conic metrics \cite{Ber} \cite{Br} \cite{CGP} \cite{Do} \cite{GP} \cite{LS} \cite{Je} \cite{JMR} \cite{LS} and the
study of conic metrics plays an important role in the solution of Yau-Tian-Donaldson conjecture \cite{Ti12} \cite{CDS1} \cite{CDS2} \cite{CDS3}.
As the $C^{0}$-estimate and $C^{2}$-estimate is not enough to obtain the solution to existence problems,, we need to get at least $C^{2,\alpha}$-
estimate to guarantee that the inverse function theorem can be applied. For conic metrics, in smooth part, it is easy to apply Calabi's
$C^{3}$-estimate as smooth case \cite{Yau}. For the conic singularities, when the cone angle $\beta<\frac{1}{2},$ even Calabi's $C^{3}$-
estimate still holds \cite{Br}. However, when $\beta\geq\frac{1}{2},$ the $C^{3}$-estimate seems not to hold any more so we can expect only a
$C^{2,\alpha}$-estimate. In conic setting, we are still not sure that whether classic Evans-Krylov-Safanov theory (see \cite{GT}) holds, as
we are not clear whether the subtle argument of choosing squares and measure estimates can be extended to conic setting. 
Actually there have been some works on extension of this classic theory to conic setting \cite{CW1} \cite{CW2} \cite{GP} \cite{W}.
When the metric is sufficiently close to the standard flat conic metric, perturbation method of the Schauder estimate can be used to derive a 
$C^{2,\alpha}$-estimate \cite{CDS2}. However we don't know the metric involved can be made sufficiently close to a standard flat conic metric
before we establish an estimate than the $C^{2}$-estimate. In \cite{Ti14}, Tian adapted the method of his PKU master thesis \cite{Ti83} to 
conic case and proved the $C^{2,\alpha}$-estimate for conical complex Monge-Ampere equations without the closeness assumption. The main idea 
of Tian's approach is to compare the second derivative to harmonic forms whose regularity was studied before.

      In this paper, we extend Tian's idea in \cite{Ti14} to derive a $C^{2,\alpha}$-estimate for conical \ka-Ricci flow, which completes the
proof of the existence theorem in \cite{Sh}. Conical \ka-Ricci flow is another interesting topic in the study of conic \ka\ metrics and there
have been some works on it \cite{CW1} \cite{CW2} \cite{LZ} \cite{MRS} \cite{Sh} \cite{W}. To establish the existence of conical \ka-Ricci flow
we need a parabolic version of $C^{2,\alpha}$-estimate. In deriving our estimate, smilar to what Tian did, we need to use regularity properties
of heat equations.

      Now let's recall the construction of unnormalized \ka-Ricci flow in \cite{Sh}. Suppose the divisor $D=\sum\limits_{i=1}^{l}(1-
\beta_{i})D_{i}(l\leq n)$ is a simply normal crossing divisor on a n dimensional compact \ka\ manifold $(M,\dr_{0}$, where each $D_{i}$ is
an irreducible divisor and $\dr_{0}$ is a smooth background \ka\ metric. Now consider unnormalized \ka-Ricci flow
\begin{equation}\label{eq:ckrf}
 \left\{ \begin{array}{rcl}
&\dt\dr=-Ric(\dr)+2\pi\sum\limits_{i=1}^{l}(1-\beta_{i})[D_{i}]\\&\dr(0)=\dr^{*}=\dr_{0}+\sum\limits_{i=1}^{l}k\ddb||S_{i}||^{2\beta_{i}}
        \end{array}\right.
\end{equation}
where $S_{i}$ is a defining holomorphic section of $[D_{i}],$ and we know that when k is a small positive number $\dr^{*}$ is a conic
\ka\ metric with cone angles $2\pi\beta_{i}$ along each $D_{i}.$ In \cite{Sh} we apply the methods of \cite{TZ} \cite{ST3} to consider
this equation in the cohomology level, i.e, write $\overline{\dr}_{t}=\dr^{*}-t(Ric(\Omega)-\sum\limits_{i=1}^{l}(1-\beta_{i})R(||\cdot||_
{i})),$ and$\dr=\overline{\dr}_{t}+\ddb\varphi,$ where $\Omega$ is a smooth volume form and $R(||\cdot||_{i}))$ is corresponding curvature
form of the line bundle $[D_{i}].$ Then the equation \eqref{eq:ckrf} can be transformed to a Monge-Ampere type equation:
\begin{equation}\label{eq:conic-ma}
 \left\{ \begin{array}{rcl}
&\dt\varphi=\log\frac{(\overline{\dr}_{t}+\ddb\varphi)^{n}}{\Omega}+\sum\limits_{i=1}^{l}\log||S_{i}||^{2(1-
\beta_{i})}\\&\varphi(\cdot,0)=0
        \end{array}\right.
\end{equation}
Recall that in \cite{Sh}, we apply the approximation techniques of \cite{CGP} \cite{GP} \cite{LZ} and estimates in \cite{TZ} \cite{ST3} to
obtain $C^{0}$-estimate and Laplacian estimate for $\varphi,$ i.e, for some uniform constant A,
\begin{equation}\label{eq:laplacian}
 A^{-1}\overline{\dr}_{0}\leq\dr=\overline{\dr}_{t}+\ddb\varphi\leq A\overline{\dr}_{0}
\end{equation}
on the time interval $[0,T]$ where $$T<T_{0}:=\sup\{t>0|[\dr_{0}]-t(c_{1}(M)-\sum\limits_{i=1}^{l}(1-\beta_{i})[D_{i}])>0\}.$$ What is missing
in \cite{Sh} is a H\"{o}lder estimate for $\ddb\varphi,$ i.e, the $C^{2,\alpha}$-estimate. We will follow the approach in \cite{Ti14}
to get the following main theorem by comparing the flow solution with the solutions to heat equations:
\begin{theorem}\label{mainthm}
 For any $\alpha\in(0,\overline{\beta}^{-1}-1,1\}),$ where $\overline{\beta}:=\max\{\beta_{1},\cdots,\beta_{l}\},$ there exist constants
$r_{0}\in(0,1)$ and $C_{\alpha}>0$ such that for any $(x_{0},t_{0})\in M\times(0,T_{0})$ and $0<r<r_{0},$
\begin{equation}\label{eq:3rd}
 \iint_{P_{r}(x_{0},t_{0})}|\nabla^{3}\varphi|^{2}\overline{\dr}_{0}^{n}dt\leq C_{\alpha}r^{2n+2\alpha},
\end{equation}
 where $P_{r}(x_{0},t_{0})=B_{r}(x_{0})\times(t_{0}-r^{2},t_{0}],$ and $B_{r}(x_{0})$ is the geodesic ball with center $x_{0}$
and radius r, and $\nabla$ is the covariant derivative, the norm is taken with respect to $\overline{\dr}_{0}.$
\end{theorem}
 Provided this theorem, by standard analysis arguments, we have the following corollary, which gives the $C^{2,\alpha}$-estimate:
\begin{corollary}\label{cor-holder}
 For any $\alpha\in(0,\min\{\overline{\beta}^{-1}-1,1\}),$ $\ddb\varphi$ is $C^{\alpha}$-bounded with respect to
$\overline{\dr}_{0}.$
\end{corollary}
       The note will be organized as below. First, begin with the Monge-Ampere type equation \eqref{eq:conic-ma}, we derive a second order
parabolic equation for complex Hessian of $\varphi.$ Then we review some facts on heat equation, which can be modified from \cite{HL}.
Finally, we follow \cite{Ti14}, to give a parabolic version of comparison argument.

\noindent{\bf Acknowledgment.} First the author wants to thank his Ph.D thesis advisor Professor Gang Tian for a lot of discussions and
encouragement. And he also wants to thank Chi Li and Zhenlei Zhang for many useful conversations. And he also thanks CSC for partial
financial support during his Ph.D career.

\section{Preliminaries}

In this section we'll obtain a parabolic equation for complex Hessian, and make some necessary preparation for the estimate. First, we
note that the Theorem \ref{mainthm} can be proved locally. As Calabi's 3rd derivative estimate can be used outside the divisor, we only
need to consider that $x_{0}\in D.$ Now choose a local coordinate system $(U;z_{1},\cdots,z_{n})$ around $x_{0}=(0,\cdots,0)$ and $D_{i}
\bigcap U=\{z_{i}=0\}.$ Now U can be identified with an open set in $\C.$ By Laplacian estimate \eqref{eq:laplacian} we have
\begin{equation}\label{eq:c2}
 C^{-1}\dr_{\beta}\leq\dr\leq C\dr_{\beta},
\end{equation}
where $\dr_{\beta}$ is the standard conic flat metric on $\C$: $$\dr_{\beta}=\sqrt{-1}\left(\sum_{i=1}^{l}\frac{dz_{i}\wedge
d\bar{z}_{i}}{|z_{i}|^{2(1-\beta_{i})}}+\sum_{i=l+1}^{n}dz_{i}\wedge d\bar{z}_{i}\right).$$ By $\ddb$-lemma (see \cite{Ti98}), we can
write $\dr=\ddb u$ for some spacetime neighborhood $U\times (t',t_{0}].$ Now consider the Monge-Ampere type equation \eqref{eq:conic-ma},
as $\ddb u=\dr=\overline{\dr}_{t}+\ddb\varphi,$ we have that
\begin{align*}
 \ddb\dt u&=-Ric(\Omega)+\sum\limits_{i=1}^{l}(1-\beta_{i})R(||\cdot||_{i})+\ddb\log\frac{(\overline{\dr}_{t}+\ddb\varphi)^{n}\prod
\limits_{i=1}^{l}||S_{i}||^{2(1-\beta_{i})}}{\Omega}\\&=\ddb(\log\frac{(\ddb u)^{n}}{\dr_{\beta}^{n}}+F),
\end{align*}
Here F is defined as $$\ddb F=\ddb\log\dr_{\beta}^{n}+\sum\limits_{i=1}^{l}(1-\beta_{i})(R(||\cdot||_{i}+\ddb\log||S_{i}||^{2}),$$
so F is a smooth function. Then we obtain a scalar equation:
\begin{equation}\label{eq:conic-potential}
\dt u=\log\frac{(\ddb u)^{n}}{\dr_{\beta}^{n}}+F.
\end{equation}
As $\overline{\dr}_{t}$ is a conic \ka\ metric which has the same singular part with the initial conic metric, it suffices to prove that
$\ddb u$ is $C^{\alpha}$-bounded. Now take the covariant derivatives of the equation \eqref{eq:conic-potential} twice with respect to
$\dr_{\beta}$, we obtain a parabolic equation for $\ddb u$:
\begin{equation}\label{eq:hessian}
.
\end{equation}
Obviously, with respect to $\dr_{\beta},$ $u_{k\bar{l}}$ is uniformly bounded. Unfortunately we can't use Krylov-Safanov estimate directly
to get a H\"{o}lder estimate. We'll make use of the properties of heat equation to derive the estimate. For this purpose, we need to do
some preparation.

      As \cite{Ti14}, we can define $\B(r)$ to be an angle domain in $\mathbb{C}^{l}\times\mathbb{C}^{n-l}$ consisting of all $(w_{1},
\cdots,w_{l},w')$ satisfying $\sum\limits_{i=1}^{l}|w_{i}|^{2}+|w'|^{2}\leq r^{2},$ where $w_{i}=\rho_{i}e^{\sqrt{-1}\theta_{i}}$ for
$\rho_{i}\in[0,r]$ and $theta_{i}\in[0,2\pi\beta_{i}], 1\leq i\leq l.$ Now we have such a Sobolev type inequality, which can be proved in
standard way(\cite{Gu} \cite{LSU}) by Sobolev inequality in \cite{Ti14}:
\begin{lemma}\label{lem-Sobolev}
 There is a constant $C_{\beta}$ which depends on $1-\beta_{i}, i=1,\cdots,l$ such that
for any smooth function h on $\B\times(-1,0]$ with boundary conditions:
\begin{equation}\label{eq:bdry}
 h(w_{1},\cdots,\rho_{i}e^{\sqrt{-1}2\pi\beta_{i}},\cdots,w_{l},w')=e^{\sqrt{-1}2\pi(1-\beta_{i})}h(w_{1},\cdots,\rho_{i},\cdots,w_{l},w')
\end{equation}
for each i, then we have
\begin{equation}\label{eq:Sobolev1}
 \left(\iint_{\B\times[-1,0]}|h|^{2(1+\frac{1}{n})}\dr_{\beta}^{n}dt\right)^{\frac{n}{n+1}}\leq C_{\beta}\left(\sup_{t}\int_{\B}|h|^{2}
\dr_{\beta}^{n}+\iint_{\B\times[-1,0]}|h|^{2}\dr_{\beta}^{n}dt\right),
\end{equation}
or for $s\in(\frac{n}{n+1},1),$
\begin{equation}\label{eq:Sobolev2}
 \iint_{\B\times[-1,0]}|h|^{2}\dr_{\beta}^{n}dt\leq C_{\beta,s}\left(\sup_{t}\int_{\B}|h|^{2s}\dr_{\beta}^{n}+
\iint_{\B\times[-1,0]}|h|^{2s}\dr_{\beta}^{n}dt\right)^{\frac{1}{s}}.
\end{equation}
\end{lemma}
Now we prove a lemma that outside the divisor D, the scalar curvature is uniformly bounded from below, which is useful in Moser's iteration
 with respect to evolving metrics and weak Harnack inequality below:
\begin{lemma}\label{lem:curvature}
Suppose the evolving conic metrics $\dr(t)$ satisfy the conical \ka-Ricci flow \eqref{eq:ckrf}, then outside the divisor D, the scalar
curvature $R\geq-C,$ which only depends on initial metric.
\end{lemma}
\begin{proof}
We know that the conclusion is true in case of smooth Ricci flow. In conic case, first, by the structure of conic \ka\ metric $\overline
{\dr}_{0},$ we can compute that $Ric(\overline{\dr}_{0})=a_{i\bar{j}}dz_{i}\wedge d\bar{z}_{j}+2\pi\sum\limits_{k=1}^{l}(1-\beta_{k})[D_{k}]
,$ where $a_{i\bar{j}}dz_{i}\wedge d\bar{z}_{j}$ is a smooth  $(1,1)$-form. From this we can see that Ricci curvature of the initial conic
metric is uniformly bounded from below, so does the scalar curvature, outside the divisor D. Outside the divisor D we have the evolution
equation for scalar curvature \cite{H} $$\dt R=\Delta R+|Ric|^{2},$$ and now we can write this equation as following:
$$(\dt-\Delta)(R-c\sum_{i=1}^{l}\log||S_{i}||^{2})=|Ric|^{2}-c\sum_{i=1}^{l}tr_{\dr}(R(||\cdot||_{i})+2\pi[D_{i}]),$$
where c is a positive constant. It's easy to see that at each time slice, the space minimum will be attained outside the divisor, so by
maximal principle \cite{H}, as $R(||\cdot||_{i}$ is uniformly bounded and $\dr\geq C\dr_{0}$ where $\dr_{0}$ is a smooth background metric,
we have that $$\min(R-c\sum_{i=1}^{l}\log||S_{i}||^{2})(t)\geq\min(R-c\sum_{i=1}^{l}\log||S_{i}||^{2})(0)-cCt.$$ Now let c tend to 0, we
obtain our result.
\end{proof}
Now make use of above two lemmas, we can prove an inequality similar to mean value inequality for subharmonic functions:
\begin{proposition}\label{prop-mean}
Suppose $v\geq 0$ satisfy that $(\dt-\Delta_{\beta})v\leq C_{0},$ or $(\dt-\Delta_{\dr})v\leq C_{0},$ where $\dr$ satisfies conical
\ka-Ricci flow \eqref{eq:ckrf} on $P(1)=B(1)\times(-1,0],$ we have that
$$\sup_{P(1/2)}v\leq C\left(\iint_{R(1)}v\dr_{\beta}^{n}dt+C_{0}\right),$$ or replace $\dr_{\beta}$ by $\dr.$
\end{proposition}
\begin{proof}
Before the $L^{1}$-control, we first prove a $L^{2}$-control that for any $\theta\in(1/2,1),$ by Moser's iteration as \cite{Gu}:
\begin{equation}\label{eq:L2}
\sup_{P(\theta)}v\leq C\left(\frac{1}{(1-\theta)^{n+1}}(\iint_{P(1)}v^{2}\dr_{\beta}^{n}dt)^{1/2}+C_{0}\right).
\end{equation}
Without loss of generality, we assume that $C_{0}=0,$ otherwise we take $v'=v-C_{0}t$ to replace v. We prove the first case. First we choose
$\eta=v\xi^{2}\chi_{[-r^{2},-\sigma]}(t)$ where $\xi$ is a cut-off function with compact support in $P(1),$ and $\chi$ is a characteristic
function. Define $P^{\sigma}(r)=B(r)\times[-r^{2},-\sigma]$ and take $\eta$ as a test function for the first inequality, we have
\begin{equation}\label{eq:parts}
 \frac{1}{2}\iint_{P^{\sigma}(r)}(\dt\xi^{2}v^{2}+\xi^{2}|\nabla v|^{2})\dr_{\beta}^{n}dt\leq C\iint_{P^{\sigma}(r)}
(|\nabla\xi|^{2}+|\xi\dot{\xi}|)v^{2}\dr_{\beta}^{n}dt.
\end{equation}
For $\theta\leq r'<r\leq 1,$ choose $\sigma\leq r'^{2}$ such that
\begin{equation}\label{eq:sup}
 \int_{B(r')}v^{2}\dr_{\beta}^{n}|_{t=-\sigma}\geq\frac{1}{2}\sup_{t\in[-r'^{2},0]}\int_{B(r')}v^{2}\dr_{\beta}^{n}.
\end{equation}
Choose $\xi(x,t)=\xi_{1}(x)\xi_{2}(t),$ where $\xi_{1}$ is 1 on $B(r'),$ and 0 outside $B(r)$ with $|\nabla\xi_{1}|\leq
\frac{C}{r-r'},$ and $\xi_{2}$ is 1 on $[-r'^{2},0],$ and 0 on $[-1,-r^{2}]$ with $|\dot{\xi}_{2}|\leq\frac{C}
{(r-r')^{2}}.$ By the equation \eqref{eq:parts}, we obtain that
\begin{align*}
\sup_{t\in[-r'^{2},0]}\int_{B(r')}v^{2}\dr_{\beta}^{n}&\leq 2\int_{B(r')}v^{2}\dr_{\beta}^{n}|_{t=-\sigma}\leq
2\iint_{P^{\sigma}(r)}\dt\xi^{2}v^{2}\dr_{\beta}^{n}dt\\&\leq\frac{C}{(r-r')^{2}}\iint_{P^{\sigma}(r)}v^{2}\dr_{\beta}^{n}dt
\leq\frac{C}{(r-r')^{2}}\iint_{P(r)}v^{2}\dr_{\beta}^{n}dt.
\end{align*}
Choose test function $\eta=v\xi^{2}\chi_{[-r^{2},0]}(t)$ then the integral domain of \eqref{eq:parts} will be $P(r)$ and
we obtain that $$\iint_{P(r')}|\nabla v|^{2}\dr_{\beta}^{n}dt\leq\frac{C}{(r-r')^{2}}\iint_{P(r)}v^{2}\dr_{\beta}^{n}dt.$$
Now by Sobolev type inequality \eqref{eq:Sobolev1}, we obtain that
$$\left(\iint_{P(r')}v^{2(1+\frac{1}{n})}\dr_{\beta}^{n}dt\right)^{\frac{n}{n+1}}\leq\frac{C}{(r-r')^{2}}
\iint_{P(r)}v^{2}\dr_{\beta}^{n}dt.$$
By computation, $(\dt-\Delta_{\beta})v^{p}\leq 0$ for $p\geq 1.$ So similar argument is true that
$$\left(\iint_{P(r')}v^{2p(1+\frac{1}{n})}\dr_{\beta}^{n}dt\right)^{\frac{n}{n+1}}\leq\frac{C}{(r-r')^{2}}
\iint_{P(r)}v^{2p}\dr_{\beta}^{n}dt.$$
Then by Moser's iteration, we obtain the inequality \eqref{eq:L2}.
     Now by \eqref{eq:L2} and Cauchy-Schwarz inequality, we have that
\begin{align*}
||v||_{L^{\infty}(P(r'))}&\leq C\left(\frac{1}{(r-r')^{n+1}}||v||_{L^{2}(P(r))}+C_{0}\right)\\&\leq C\left(\frac{1}
{(r-r')^{n+1}})||v||_{L^{\infty}(P(r))}^{1/2}||v||_{L^{1}(P(r))}^{1/2}+C_{0}\right)\\&\leq\frac{||v||_
{L^{\infty}(P(r))}}{2}+C\left(\frac{1}{(r-r')^{2n+2}}||v||_{L^{1}(P(r))}+C_{0}\right),
\end{align*}
then by an iteration lemma (Lemma 4.3 on Page 75 of \cite{HL}), we get that
$$||v||_{L^{\infty}(P(\theta))}\leq C\left(\frac{1}{(1-\theta)^{2n+2}}||v||_{L^{1}(P(r))}+C_{0}\right)$$
which completes the proof of the first case. For the proof of the second case, note that in the derivation of similar
inequality to \eqref{eq:parts}, there is an extra term $-R\xi^{2}v^{2}$ in the right hand side integral which
comes from the derivative of $\dr^{n}.$ However by Lemma \ref{lem:curvature} this term is bounded from above by
$C\xi^{2}v^{2},$ then all the argument for the first case follows.
\end{proof}
From this proposition, we can easily get following corollaries:
\begin{corollary}\label{cor-heat-energy}
For the function h which solves heat equation $(\dt-\Delta)h=0$ and satifies the same boundary condition \eqref{eq:bdry} of Lemma
\ref{lem-Sobolev} on $B_{\beta}\times(-1,0],$ and any $r<1,$ we have that
\begin{equation}\label{eq:heat-angle-mean}
 \iint_{B_{\beta}(r)\times(-r^{2},0]}|\nabla h|^{2}dw\wedge d\bar{w}dt\leq Cr^{2n-2+2\overline{\beta}^{-1})}\iint_{B_{\beta}
\times(-1,0]}|\nabla h|^{2}dw\wedge d\bar{w}dt.
\end{equation}
Moreover, for any $\sigma<r<1,$ and function v which solves $(\dt-\Delta_{\beta})v=0$ on $P(1),$ we have that
\begin{equation}\label{eq:heat-mean}
 \iint_{P(\sigma)}|\nabla v|^{2}\dr_{\beta}^{n}dt\leq C(\frac{\sigma}{r})^{2n-2+2\overline{\beta}^{-1})}\iint_{P(r)}|\nabla v|^{2}
\dr_{\beta}^{n}dt.
\end{equation}
\end{corollary}
\begin{proof}
 Let's see what the boundary condition \eqref{eq:bdry} means. We choose coordinate change $z_{i}=(\beta_{i}w_{i})^{\frac{1}{\beta_{i}}}$
for $i=1\cdots,l$ and $z_{i}=w_{i}$ for other i, which maps the angle domain $B_{\beta}$ to the ball with standard conic metric. Then we
find that
\begin{align*}
 &\frac{\partial h(w(z),t)}{\partial z_{i}}(\cdots,z_{i}(\rho_{i}e^{\sqrt{-1}2\pi\beta_{i}}),\cdots)=\frac{\partial h(w,t)}{\partial w_{i}}
\frac{\partial w_{i}}{\partial z_{i}}(\cdots,\rho_{i}e^{\sqrt{-1}2\pi\beta_{i}},\cdots)\\=&e^{\sqrt{-1}2\pi(1-\beta_{i})}
\frac{\partial h(w,t)}{\partial w_{i}}(\cdots,\rho_{i},\cdots)(\beta_{i}w_{i})^{1-\beta_{i}^{-1}}(\rho_{i}e^{\sqrt{-1}2\pi\beta_{i}})
\\=&e^{\sqrt{-1}2\pi(1-\beta_{i})}\frac{\partial h(w,t)}{\partial w_{i}}\frac{\partial w_{i}}{\partial z_{i}}(\cdots,\rho_{i},\cdots)
e^{\sqrt{-1}2\pi(\beta_{i}-1)}\\=&\frac{\partial h(w,t)}{\partial w_{i}}\frac{\partial w_{i}}{\partial z_{i}}(\cdots,\rho_{i},\cdots)
=\frac{\partial h(w(z),t)}{\partial z_{i}}(\cdots,z_{i}(\rho_{i}),\cdots),
\end{align*}
which allows us to consider h as a function on the ball with standard conic metric. Note that $\dr_{\beta}=\sqrt{-1}\sum dw_{i}\wedge
d\bar{w}_{i},$ So \eqref{eq:heat-mean} follows from \eqref{eq:heat-angle-mean} and we only need to prove \eqref{eq:heat-angle-mean}.

         Note that the functions $\zeta_{i}=w_{i}^{\frac{k}{\beta_{i}-1}}$ satisfy boundary condition \eqref{eq:bdry} and if we write
h as a function of $\zeta_{1},\cdots,\zeta_{l},w_{l+1},\cdots,w_{n},t$ we can find that $h(\cdots,\zeta_{i}e^{\sqrt{-1}2\pi\beta_{i}},
\cdots)=e^{\sqrt{-1}2\pi\beta_{i}}h(\cdots,\zeta_{i},\cdots),$ then we can extend h to be a function on a ball which solves the heat
solution with respect to Euclidean metric. Note that $(\dt-\Delta)|\frac{\partial h(w(\zeta),t)}{\partial\zeta_{i}}|^{2}\leq 0,$ and
$$\frac{\partial h(w(\zeta),t)}{\partial\zeta_{i}}=\frac{\beta}{k-\beta}\frac{\partial h(w,t)}{\partial w_{i}}w_{i}^{2-\frac{k}{\beta}},$$
by Proposition \ref{prop-mean}, $$\sup_{P(r)}|\frac{\partial h(w,t)}{\partial w_{i}}|^{2}\leq r^{2(\frac{k}{\beta}-2)}\iint_{B_{\beta}
\times[-1,0]}|\frac{\partial h(w,t)}{\partial w_{i}}|^{2}|w_{i}|^{2(2-\frac{k}{\beta})}dw\wedge d\bar{w}dt.$$
As $k\geq 1,$ for all i, choose the smallest exponential $2(\min(\beta_{i}^{-1})-2)=2(\overline{\beta}^{-1}-2)$ and integrate on $P(r),$
we obtain our result.
\end{proof}
        On the other hand, we still need a weak Harnack inequality as Theorem 4.15 in \cite{HL}. Here we need to consider a new spacetime
neighborhood $P'(r)=B(r)\times(-r^{2},r^{2})$ to replace $P(r).$ Then we will have
\begin{proposition}\label{prop-weak harnack}
 Suppose $\dr=\ddb u$ solves the conical \ka-Ricci flow \eqref{eq:ckrf} on $M\times[-T,T](T>1),$ and a nonnegative function
$v\in H^{1}(P'(r))$ satisfies
\begin{equation}\label{eq:supersolution}
 (\dt-\Delta')v\geq-k,
\end{equation}
where $\Delta'$ denotes the Laplacian with respect to $\dr.$ Then we will have
\begin{equation}\label{eq:weak harnack}
 \iint_{P'(1)}v\dr_{\beta}^{n}dt\leq C(\inf_{P'(\frac{1}{2})}v+k).
\end{equation}
\end{proposition}
\begin{proof}
 The proof is similar to the elliptic version of weak Harnack inequality in Theorem 4.15 of \cite{HL} so we just sketch the proof.
We set $\overline{v}=v+k$ and $w=\overline{v}^{-1}.$ Choose a cut-off function $\xi(x,t)=\xi_{1}(x)\xi_{2}(t),$ where $\xi_{1}$ is 1 on
$B(\frac{1}{2}),$ and 0 near $\partial B(1)$ with $|\nabla\xi_{1}|\leq 4,$ and $\xi_{2}$ is 1 on $[-\frac{1}{4},\frac{1}{4}],$ and 0 near
$t=\pm 1$ with $|\dot{\xi}_{2}|\leq\ 16.$ Take $\overline{v}^{-2}\xi$ as the test function of \eqref{eq:supersolution} we obtain that
$$\iint_{P'(1)}\xi(\Delta'-\dt)w\dr^{n}dt\leq \iint_{P'(1)}\xi\frac{C}{\overline{v}^{2}}\dr^{n}dt.$$ By Proposition
\ref{prop-mean} we obtain that $$\sup_{P'(\frac{1}{2})}\overline{v}^{-p}\leq C\iint_{P'(1)}\overline{v}^{-p}\dr^{n}dt\leq C\iint_{P'(1)}
\overline{v}^{-p}\dr_{\beta}^{n}dt,$$ for any $p>0,$ which indicates that $$\inf_{P'(\frac{1}{2})}v+k\geq C\left(\iint_{P'(1)}
\overline{v}^{p}\dr_{\beta}^{n}dt\iint_{P'(1)}\overline{v}^{-p}\dr_{\beta}^{n}dt\right)^{-1/p}(\iint_{P'(1)}v\dr_{\beta}^{n}dt)^{1/p}.$$
Now we need to prove $$\iint_{P'(1)}\overline{v}^{p}\dr_{\beta}^{n}dt\iint_{P'(1)}\overline{v}^{-p}\dr_{\beta}^{n}dt\leq C$$ for some
$p=p_{0}>0.$ It suffices to prove that $$\iint_{P'(1)}e^{p_{0}|h|}\dr_{\beta}^{n}dt\leq C,$$ where $h=\log\overline{v}
-\frac{1}{|P'(1)|_{\dr_{\beta}}}\iint_{P'(1)}\overline{v}\dr_{\beta}^{n}dt.$ As $e^{p_{0}|h|}=1+|p_{0}h|+\frac{|p_{0}h|^{2}}{2!}+\cdots
+\frac{|p_{0}h|^{l}}{l!}+\cdots$ it suffices to bound $\iint_{P'(1)}|h|^{l}\dr_{\beta}^{n}dt$ for any positive integer $l$.

Denote $det(u_{k\bar{l}})u^{i\bar{j}}$ by $U^{i\bar{j}}$ and by \cite{Ti14}, we have $U_{i}^{i\bar{j}}=0$ (See also Lemma \ref{lem-3.1}
in the next section). Now take $\overline{v}^{-1}\xi^{2}$ as the test function in \eqref{eq:supersolution}, we have that
\begin{align*}
&-\iint_{P'(1)}\frac{k}{\overline{v}}\xi^{2}\dr_{\beta}^{n}dt\leq\iint_{P'(1)}\overline{v}^{-1}\xi^{2}
(\dt-\Delta')\overline{v}\dr_{\beta}^{n}dt\\=&\iint_{P'(1)}\xi^{2}\left(\dt h-\overline{v}^{-1}U^{i\bar{j}}\overline{v}_{i\bar{j}}
det(u_{k\bar{l}}^{-1})\right)\dr_{\beta}^{n}dt\\=&\iint_{P'(1)}\left(-2h\xi\dot{\xi}+2\xi\xi_{\bar{j}}\overline{v}^{-1}U^{i\bar{j}}
\overline{v}_{i} det(u_{k\bar{l}}^{-1})+\xi^{2}(-\overline{v}^{-2}U^{i\bar{j}}\overline{v}_{i}\overline{v}_{\bar{j}}det(u_{k\bar{l}}^{-1})
+U^{i\bar{j}}h_{i}det(u_{k\bar{l}}^{-1})_{\bar{j}})\right)\dr_{\beta}^{n}dt \\=&\iint_{P'(1)}\left(-2h\xi\dot{\xi}+2\xi\xi_{\bar{j}}h_{i}
u^{i\bar{j}}-\xi^{2}|\nabla'h|^{2}-hU^{i\bar{j}}(2\xi\xi_{i}det(u_{k\bar{l}}^{-1})_{\bar{j}}+\xi^{2}det(u_{k\bar{l}}^{-1})_{i\bar{j}})
\right)\dr_{\beta}^{n}dt
\end{align*}
For the last term, we have that
\begin{align*}
-U^{i\bar{j}}det(u_{k\bar{l}}^{-1})_{i\bar{j}}&=U^{i\bar{j}}\left(\frac{(\log det(u_{k\bar{l}}))_{i}}{det(u_{k\bar{l}})}\right)_{\bar{j}}
\\&=u^{i\bar{j}}((\log det(u_{k\bar{l}}))_{i\bar{j}}-det(u_{k\bar{l}})^{-2}det(u_{k\bar{l}})_{i}det(u_{k\bar{l}})_{\bar{j}})\\&=-R(\dr)-
|\nabla'detdet(u_{k\bar{l}})|^{2}\leq C-|\nabla'det(u_{k\bar{l}})|^{2},
\end{align*}
where the last inequality follows from Lemma \ref{lem:curvature}. Make use of this and Cauchy-Schwarz inequality, we have
\begin{equation}\label{eq:iteration1}
\iint_{P'(1)}\xi^{2}|\nabla'h|^{2}\dr_{\beta}^{n}dt\leq C\iint_{P'(1)}(|\nabla'\xi|^{2}+k\overline{v}^{-1}\xi^{2}-2h\xi\dot{\xi}+C(\xi^{2}
+|\nabla\xi|^{2})h)\dr_{\beta}^{n}dt,
\end{equation}
where $\nabla'$ denotes the gradient with respect to $\dr,$ which also have corresponding Sobolev inequality and Poincare inequality as it
is equivalent to $\dr_{\beta}.$ Note that in the new spacetime domain $P'(1),$ the first term in the Sobolev
inequality \eqref{eq:Sobolev1} can be canceled and by Poincare inequality we have that $$\iint_{P'(1)}|h|^{2}\dr_{\beta}^{n}dt\leq C.$$

         For $l>2,$ take $\overline{v}^{-1}\xi^{2}h^{2l}$ as the test function in \eqref{eq:supersolution} and the similar estimate above,
we have that
\begin{align*}
\iint_{P'(1)}\xi^{2}h^{2l}|\nabla'h|^{2}\dr_{\beta}^{n}dt&\leq\iint_{P'(1)}(2l\xi^{2}h^{2l-1}|\nabla'h|^{2}+k\overline{v}^{-1}\xi^{2}h^{2l}
\\&+2\xi\nabla'\xi\cdot h^{2l}\nabla'h+\frac{1}{2l+1}(C(|\nabla\xi|^{2}+\xi^{2})h^{2l+1}-2\xi\dot{\xi}h^{2l+1}))\dr_{\beta}^{n}dt.
\end{align*}
Apply Cauchy-Schwarz inequality and Young's inequality, we obtain that
$$\iint_{P'(1)}|\nabla'(\xi h^{l+1})|^{2}\dr_{\beta}^{n}dt\leq C\iint_{P'(1)}\left((2l)^{2l+2}\xi^{2}|\nabla'h|^{2}+(2l)^{\alpha}
(|\nabla'\xi|^{2}+|\dot{\xi}|+C)|h|^{2(l+1)}\right)\dr_{\beta}^{n}dt,$$ where $\alpha>0.$ Then follow \cite{HL} to use Moser's iteration,
we can prove that there exists $p_{0}>0$ such that $\iint_{P'(1)}e^{p_{0}|h|}\dr_{\beta}^{n}dt\leq C.$

         Finally, we need to prove that the inequality \eqref{eq:weak harnack} is true for $p=1.$ We only need to take $\xi^{2}\overline{v}
^{-\alpha}$ as the test function to \eqref{eq:supersolution} where $\alpha\in(0,1),$ then we have
$$\frac{2\alpha}{(1-\alpha)^{2}}\iint_{P'(1)}\xi^{2}|\nabla'\overline{v}^{\frac{1-\alpha}{2}}|^{2}\dr_{\beta}^{n}dt\leq\iint_{P'(1)}(\frac
{C(|\nabla\xi|^{2}+\xi^{2})-2\xi\dot{\xi}}{1-\alpha}+\xi^{2}+\frac{2}{\alpha}|\nabla'\xi|^{2})\overline{v}^{1-\alpha}\dr_{\beta}^{n}dt.$$
As \cite{HL}, the proposition follows from Sobolev inequality and finite steps of iterations.
\end{proof}

        We end this section by a lemma which indicates that Corollary \ref{cor-holder} follows from Theorem \ref{mainthm}:
\begin{lemma}\label{lem-local integral}
 Suppose $u\in H_{loc}^{1}(P(1))$ satisfies $$\iint_{P(r)}|\nabla u|^{2}\dr_{\beta}^{n}dt\leq Cr^{2n+2\alpha},$$ then $u\in C^{\alpha}$
at the origin.
\end{lemma}
\begin{proof}
 We only need to note that by smallest square theorem and Poincare Inequality,
\begin{align*}
 \iint_{P(r)}|u-\frac{1}{|P(r)|}\iint_{P(r)}u\dr_{\beta}^{n}dt|^{2}\dr_{\beta}^{n}dt\leq&\int_{-r^{2}}^{0}\int_{B(r)\times\{t\}}
|u-\frac{1}{|B(r)|}\int_{B(r)\times\{t\}}u\dr_{\beta}^{n}|^{2}\dr_{\beta}^{n}dt\\ \leq&Cr^{2}\iint_{P(r)}|\nabla u|^{2}\dr_{\beta}^{n}dt
\leq Cr^{2n+2+2\alpha},
\end{align*}
then this lemma follows from the same argument of Theorem 3.1 on Page 48 of \cite{HL}.
\end{proof}

\section{Proof of main theorem \ref{mainthm}}

       We begin the proof of main theorem \ref{mainthm} along Tian's second approach in \cite{Ti14}. We denote $det(u_{k\bar{l}})
u^{i\bar{j}}$ by $U^{i\bar{j}}$ and recall two lemmas in \cite{Ti14}:
\begin{lemma}\label{lem-3.1}
 For each j, $U_{i}^{i\bar{j}}=0.$
\end{lemma}
\begin{proof}
$$  U_{i}^{i\bar{j}}=(det(u_{k\bar{l}})u^{i\bar{j}})_{i}=det(u_{k\bar{l}})(u^{k\bar{l}}u_{k\bar{l}i}u^{i\bar{j}}-u^{i\bar{l}}
u^{k\bar{j}}u_{k\bar{l}i})=0,$$
as we can switch $j$ and $l$.
\end{proof}
\begin{lemma}\label{lem-3.2}
For any positive definite Hermitian matrix $(\lambda_{i\bar{j}}),$ we have
 \begin{equation}\label{eq:second order}
  |det(\lambda_{p\bar{q}})\lambda^{i\bar{j}}u_{i\bar{j}}-det(u_{p\bar{q}})-(n-1)det(\lambda_{p\bar{q}})|\leq
C\sum_{i,j}|u_{i\bar{j}}-\lambda_{i\bar{j}}|^{2},
 \end{equation}
where $\lambda^{i\bar{j}}$ denotes the inverse of $\Lambda=(\lambda_{i\bar{j}})$ and C is a constant depending only on the norms of
$\Lambda$ and $u_{i\bar{j}}.$
\end{lemma}
\begin{proof}
 Consider the function f of matrix $u_{i\bar{j}}$ that $f(u_{i\bar{j}})=det(\lambda_{p\bar{q}})\lambda^{i\bar{j}}u_{i\bar{j}}-
det(u_{p\bar{q}})-(n-1)det(\lambda_{p\bar{q}})$ at the point $\Lambda=(\lambda_{i\bar{j}})$ that $f(\lambda_{i\bar{j}})=0$
and $$\frac{\partial f}{\partial u_{i\bar{j}}}|_{u_{i\bar{j}}=\lambda_{i\bar{j}}}=det(\lambda_{p\bar{q}})\lambda^{i\bar{j}}-
det(u_{p\bar{q}})u^{i\bar{j}}|_{u_{i\bar{j}}=\lambda_{i\bar{j}}}=0.$$
Then this lemma follows from Taylor expansion of f at point $u_{i\bar{j}}=\lambda_{i\bar{j}}.$
\end{proof}

       From now on we will modify Tian's arguments in \cite{Ti14} to parabolic case. First, we have
\begin{lemma}\label{lem-3.3}
 There exsits some $q>2$ which may depend on $\beta,||u_{i\bar{j}}||_{L^{\infty}}$ and $||F_{i\bar{j}}||_{L^{\infty}}$ such that
for any $P(2r)\subset U\times(-1,0]$ around $(y,0),$ we have
\begin{equation}\label{eq:Gehring}
 \left(\frac{1}{r^{2n+2}}\iint_{P(r)}(1+|\nabla\dr|)^{q}\dr_{\beta}^{n}dt\right)^{\frac{2}{q}}\leq\frac{C}{r^{2n+2}}\iint_{P(2r)}(1+
|\nabla\dr|)^{2}\dr_{\beta}^{n}dt,
\end{equation}
where C denotes a uniform constant.
\end{lemma}
\begin{proof}
 First assume that $y=x_{0}$ and move $t_{0}$ to be 0. Define $\lambda_{i\bar{j}}:=r^{-(2n+2)}\iint_{P(r)}u_{i\bar{j}}\dr_{\beta}^{n}dt$
for $i,j=1,\cdots,n.$ By using unitary transformation if necessary, we assume that $\lambda_{i\bar{j}}=0$ if $i\neq j$ and $i,j>l.$ By
Laplacian estimate \eqref{eq:c2}, we have $A^{-1}\mathbf{I}\leq \Lambda=(\lambda_{i\bar{j}})\leq A\mathbf{I}.$ Choose a cutoff function
$\eta=\eta_{1}(x)\eta_{2}(t)$ on $P(r)$ such that $\eta_{1}(x)=1$ on $B(y,\frac{3r}{4}),\eta_{1}(x)=0$ near $\partial B(y,1),$
and $\eta_{2}(t)=1$ on $[-\frac{9r^{2}}{16},0],\eta_{2}(t)=0$ near $t=-r^{2},$ with $|\nabla\eta_{1}|^{2},|\nabla^{2}\eta_{1}|,|\dot
{\eta_{2}}|\leq\frac{C}{r^{2}}.$ Make use of the equation \eqref{eq:hessian} and Lemma \ref{lem-3.1}, denote $\lambda_{k}=\lambda_{k
\bar{k}}$ and $\lambda=\lambda_{1}\cdots\lambda_{n},$ we have
\begin{align*}
 &\iint_{P(r)}U^{i\bar{j}}\left(\sum_{k=1}^{n}\frac{\lambda}{\lambda_{k}}u_{k\bar{k}}-det(u_{p\bar{q}})-(n-1)\lambda\right)\eta_{i\bar{j}}
\dr_{\beta}^{n}dt\\=&\iint_{P(r)}U^{i\bar{j}}\left((\sum_{k=1}^{n}\frac{\lambda}{\lambda_{k}}u_{k\bar{k}i\bar{j}})\eta-det(u_{p\bar{q}})
\eta_{i\bar{j}}\right)\dr_{\beta}^{n}dt\\ \geq&\iint_{P(r)}\sum_{k=1}^{n}\frac{\lambda}{\lambda_{k}}det(u_{p\bar{q}})(u^{i\bar{q}}
u^{p\bar{j}}u_{i\bar{j}k}u_{p\bar{q}\bar{k}}+\dt u_{k\bar{k}}-f_{k\bar{k}})\eta\dr_{\beta}^{n}dt-Cr^{2n}\\ \geq&c\iint_{P(r)}\eta|\nabla
\dr|^{2}\dr_{\beta}^{n}dt+\int_{B(r)\times\{0\}}\sum_{k=1}^{n}\frac{\lambda}{\lambda_{k}}det(u_{p\bar{q}})u_{k\bar{k}}\eta\dr_{\beta}^{n}
\\&-\iint_{P(r)}\sum_{k=1}^{n}\frac{\lambda}{\lambda_{k}}u_{k\bar{k}}det(u_{p\bar{q}})(\dot{\eta}+\eta\dt(\log det(u_{p\bar{q}})))\dr_
{\beta}^{n}dt-Cr^{2n}\\ \geq&c\iint_{P(r)}\eta|\nabla\dr|^{2}\dr_{\beta}^{n}dt+\iint_{P(r)}\sum_{k=1}^{n}\frac{\lambda}{\lambda_{k}}
u_{k\bar{k}}det(u_{p\bar{q}})(R-\dot{\eta})\dr_{\beta}^{n}dt-Cr^{2n}\\ \geq&c\iint_{P(r)}\eta|\nabla\dr|^{2}\dr_{\beta}^{n}dt-Cr^{2n},
\end{align*}
where the last inequality uses the Lemma \ref{lem:curvature} that the scalar curvature stays uniformly bounded from below.
By Lemma \ref{lem-3.2}, we have $$\iint_{P(\frac{3r}{4})}|\nabla\dr|^{2}\dr_{\beta}^{n}dt\leq C\left(r^{2n}+r^{-2}\iint_{P(r)}\sum_{i,j}
|u_{i\bar{j}}-\lambda_{i\bar{j}}|^{2}\dr_{\beta}^{n}dt\right).$$ By Sobolev inequality \eqref{eq:Sobolev2} in Lemma \ref{lem-Sobolev},
\begin{equation}\label{eq:de-expo}
 \iint_{P(\frac{3r}{4})}(1+|\nabla\dr|)^{2}\dr_{\beta}^{n}dt\leq Cr^{(2n+2)(1-\frac{1}{s})}\left(\iint_{P(r)}(1+|\nabla\dr|)^{2s}
\dr_{\beta}^{n}dt\right)^{\frac{1}{s}},
\end{equation}
where $s\in(\frac{n}{n+1},1).$ This inequality holds for parabolic neighborhoods with center point on or not on the divisor. Then by a
covering argument that for any parabolic neighborhood $P((y,t),2r)\subset U\times (-1,0],$
$$\frac{1}{r^{2n+2}}\iint_{P(r)}(1+|\nabla\dr|)^{2}\dr_{\beta}^{n}dt\leq C\left(\frac{1}{r^{2n+2}}\iint_{P(2r)}(1+|\nabla\dr|)^{2s}
\dr_{\beta}^{n}dt\right)^{\frac{1}{s}}.$$
Then \eqref{eq:Gehring} follows from Gehring's inverse H\"{o}lder inequality \cite{Ge}.
\end{proof}
      Similar to \cite{Ti14}, we can define a $(1,1)$-form v solving a heat equation:
\begin{equation}\label{eq:heat}
 \dt v=\sum_{i=1}^{n}\frac{1}{\lambda_{i}}v_{i\bar{i}}\;on\;P(r)\quad v=\dr\;on\;B(r)\times\{-r^{2}\}
\bigcup\partial B(r)\times(-r^{2},0].
\end{equation}
Then we have such a lemma:
\begin{lemma}\label{lem-3.4}
 For any $P((y,t),4r)\subset U\times (-1,0],$ and $\sigma<r,$ we have
\begin{align}\label{eq:local control}
 &\iint_{P(\sigma)}|\nabla\dr|^{2}\dr_{\beta}^{n}dt-cr^{2n+2}\\ \leq& c\left[(\frac{\sigma}{r})^{2n-2+2\overline{\beta}^{-1}}+\left(\frac{1}
{r^{2n}}\iint_{P(\sigma)}(1+|\nabla\dr|^{2})\dr_{\beta}^{n}dt\right)^{\frac{q-2}{q}}\right]\iint_{P(r)}|\nabla\dr|^{2}
\dr_{\beta}^{n}dt.\nonumber
\end{align}
\end{lemma}
\begin{proof}
 Take $\hat{\dr}=\dr-v$ as the test function for the Hessian equation \eqref{eq:hessian}, considering \eqref{eq:heat}, integrate by parts,
we have
\begin{align*}
 &-\iint_{P(r)}\hat{\dr}_{l\bar{k}}(u^{i\bar{q}}u^{p\bar{j}}u_{i\bar{j}k}u_{p\bar{q}\bar{l}}-F_{k\bar{l}})\dr_{\beta}^{n}dt\\
=&\iint_{P(r)}\hat{\dr}_{l\bar{k}}(\dt u_{k\bar{l}}-u^{i\bar{j}}u_{k\bar{l}i\bar{j}})\dr_{\beta}^{n}dt\\=&\iint_{P(r)}\hat{\dr}_{l\bar{k}}
(\dt\hat{\dr}_{k\bar{l}}+\sum_{i=1}^{n}\frac{1}{\lambda_{i}}v_{i\bar{i}k\bar{l}}-u^{i\bar{j}}u_{k\bar{l}i\bar{j}})\dr_{\beta}^{n}dt\\
=&\iint_{P(r)}\left(\frac{1}{2}\dt|\hat{\dr}|^{2}+\hat{\dr}_{l\bar{k}}(-\sum_{i=1}^{n}\frac{1}{\lambda_{i}}\hat{\dr}_{i\bar{i}
k\bar{l}}+(\frac{1}{\lambda_{i}}\delta_{ij}-u^{i\bar{j}})u_{k\bar{l}i\bar{j}})\right)\dr_{\beta}^{n}dt\\=&\int_{B(r)\times\{0\}}\frac
{|\hat{\dr}|^{2}}{2}\dr_{\beta}^{n}+\iint_{P(r)}\sum_{i=1}^{n}\frac{1}{\lambda_{i}}|\hat{\dr}_{k\bar{l}i}|^{2}\dr_{\beta}^{n}dt
\\&-\iint_{P(r)}\left((\frac{1}{\lambda_{i}}\delta_{ij}-u^{i\bar{j}})\hat{\dr}_{l\bar{k}\bar{j}}u_{k\bar{l}i}+\hat{\dr}_{l\bar{k}}
u^{i\bar{q}}u^{p\bar{j}}u_{k\bar{l}i}u_{p\bar{q}\bar{j}}\right)\dr_{\beta}^{n}dt,
\end{align*}
then by Cauchy-Schwarz inequality, and note that $A^{-1}-B^{-1}=A^{-1}(B-A)B^{-1},$ we obtain that
\begin{equation}\label{eq:lem-3.4-1}
 \iint_{P(r)}|\nabla\hat{\dr}|^{2}\dr_{\beta}^{n}dt\leq c\left(r^{2n+2}+\iint_{P(r)}(|\hat{\dr}|+|u_{i\bar{j}}-\lambda_{i}\delta_{ij}|^{2})
|\nabla\dr|^{2}\dr_{\beta}^{n}dt\right).
\end{equation}
By Lemma \ref{lem-3.3} and Poincare inequality, we have
\begin{align}\label{eq:lem-3.4-2}
&\iint_{P(r)}|u_{i\bar{j}}-\lambda_{i}\delta_{ij}|^{2}|\nabla\dr|^{2}\dr_{\beta}^{n}dt\\ \leq& r^{2n+2}\left(r^{-2n-2}\iint_{P(r)}|\nabla
\dr|^{q}\dr_{\beta}^{n}dt\right)^{\frac{2}{q}}\left(r^{-2n-2}\iint_{P(r)}|u_{i\bar{j}}-\lambda_{i}\delta_{ij}|^{\frac{2q}{q-2}}\dr_{\beta}
^{n}dt\right)^{\frac{q-2}{q}}\nonumber\\ \leq&C\left(r^{-2n-2}\iint_{P(r)}|u_{i\bar{j}}-\lambda_{i}\delta_{ij}|^{2}\dr_{\beta}^{n}dt\right)
^{\frac{q-2}{q}}\iint_{P(r)}(1+|\nabla\dr|)^{2}\dr_{\beta}^{n}dt\nonumber\\ \leq&C\left(r^{-2n}\iint_{P(r)}|\nabla\dr|^{2}\dr_{\beta}^{n}dt
\right)^{\frac{q-2}{q}}\iint_{P(r)}(1+|\nabla\dr|)^{2}\dr_{\beta}^{n}dt.\nonumber
\end{align}
For the other part on the right hand side of \eqref{eq:lem-3.4-1}, we may assume $\frac{q}{q-2}>2,$ otherwise the following second
inequality follows from H\"{o}lder inequality:
\begin{align}\label{eq:lem-3.4-3}
 &\iint_{P(r)}|\hat{\dr}||\nabla\dr|^{2}\dr_{\beta}^{n}dt\\ \leq& r^{2n+2}\left(r^{-2n-2}\iint_{P(r)}|\nabla\dr|^{q}\dr_{\beta}^{n}dt
\right)^{\frac{2}{q}}\left(r^{-2n-2}\iint_{P(r)}|\hat{\dr}|^{\frac{q}{q-2}}\dr_{\beta}^{n}dt\right)^{\frac{q-2}{q}}\nonumber\\ \leq&C
\left(r^{-2n-2}\iint_{P(r)}|\hat{\dr}|^{2}\dr_{\beta}^{n}dt\right)^{\frac{q-2}{q}}\iint_{P(r)}(1+|\nabla\dr|)^{2}\dr_{\beta}^{n}dt\nonumber
\\ \leq&C\left(r^{-2n}\iint_{P(r)}|\nabla\hat{\dr}|^{2}\dr_{\beta}^{n}dt\right)^{\frac{q-2}{q}}\iint_{P(r)}(1+|\nabla\dr|)^{2}\dr_{\beta}
^{n}dt\nonumber\\ \leq&C\left(r^{-2n}\iint_{P(r)}(1+|\nabla\dr|^{2})\dr_{\beta}^{n}dt\right)^{\frac{q-2}{q}}\iint_{P(r)}
(1+|\nabla\dr|)^{2}\dr_{\beta}^{n}dt,\nonumber
\end{align}
where the last inequality follows from \eqref{eq:lem-3.4-1}.
Now apply Corollary \ref{cor-heat-energy} to v, we have that
\begin{align}\label{eq:lem-3.4-4}
 \iint_{P(\sigma)}|\nabla\dr|^{2}\dr_{\beta}^{n}dt&\leq\iint_{P(\sigma)}(|\nabla\hat{\dr}|^{2}+|\nabla v|^{2})\dr_{\beta}^{n}dt\\&
\leq c(\frac{\sigma}{r})^{2n-2+2\overline{\beta}^{-1}}\iint_{P(r)}|\nabla v|^{2}\dr_{\beta}^{n}dt+\iint_{P(r)}|\nabla\hat{\dr}|^{2}
\dr_{\beta}^{n}dt\nonumber\\&\leq c(\frac{\sigma}{r})^{2n-2+2\overline{\beta}^{-1}}\iint_{P(r)}(|\nabla\dr|^{2}+|\nabla\hat{\dr}|)
\dr_{\beta}^{n}dt+\iint_{P(r)}|\nabla\hat{\dr}|^{2}\dr_{\beta}^{n}dt\nonumber\\&\leq c(\frac{\sigma}{r})^{2n-2+2\overline{\beta}^{-1}}
\iint_{P(r)}|\nabla\dr|^{2}\dr_{\beta}^{n}dt+C\iint_{P(r)}|\nabla\hat{\dr}|^{2}\dr_{\beta}^{n}dt,\nonumber
\end{align}
Now combine \eqref{eq:lem-3.4-1} to \eqref{eq:lem-3.4-4}, \eqref{eq:local control} follows.
\end{proof}

      To conclude the main theorem, similar to \cite{Ti14}, we still need a lemma which controls the right hand side terms in
\eqref{eq:local control}:
\begin{lemma}\label{lem-small energy}
 For any $\epsilon_{0}>0,$ there is an $l$ depending only on $\epsilon_{0},||\Delta u||_{L^{\infty}},||\Delta F||_{L^{\infty}}$ satisfying:
For any $\tilde{r}>0$ with $P((y,t),\tilde{r})\subset U\times(-1,0],$ there exists $r\in[2^{-l}\tilde{r},\tilde{r}]$ such that
\begin{equation}\label{eq:small energy}
 r^{-2n}\iint_{P(r)}(1+|\nabla\dr|^{2})\dr_{\beta}^{n}dt\leq\epsilon_{0}.
\end{equation}
\end{lemma}
\begin{proof}
 Actually we only need to consider the integral of $|\nabla\dr|^{2}$ and use the domain $P'(r)$ to replace the smaller $P(r).$ Now by the
Hessian equation \eqref{eq:hessian} again, we have
\begin{equation}\label{eq:double Laplacian}
 (\Delta'-\dt)\Delta u=u^{i\bar{q}}u^{p\bar{j}}g^{k\bar{l}}u_{i\bar{j}k}u_{p\bar{q}\bar{l}}-\Delta F,
\end{equation}
where $\Delta'$ denotes the Laplacian of $\dr,$ and $g$ denotes the metric of $\dr_{\beta}.$

       Now denote $M_{r}=\sup_{P'(r)}\Delta u,$ and choose a cut-off function
$\eta=\eta_{1}(x)\eta_{2}(t)$ on $P(r)$ such that $\eta_{1}(x)=1$ on $B(y,\frac{r}{2}),\eta_{1}(x)=0$ near $\partial B(y,1),$ and
$\eta_{2}(t)=1$ on $[-\frac{r^{2}}{4},0],\eta_{2}(t)=0$ near $t=-r^{2},$ with $|\nabla\eta_{1}|^{2},|\nabla^{2}\eta_{1}|,|\dot{\eta_{2}}|
\leq\frac{C}{r^{2}}.$ Make use of the estimate in \eqref{eq:iteration1} of Propositon \ref{prop-weak harnack}, we have
\begin{align*}
 \iint_{P'(r)}\eta^{2}|\nabla\dr|^{2}\dr_{\beta}^{n}dt&\leq Cr^{2n+2}+C\iint_{P'(r)}\eta^{2}(\dt-\Delta')(M_{r}-\Delta u)\dr_{\beta}^{n}dt
\\&\leq Cr^{2n+2}+C\iint_{P'(r)}(-2\eta\dot{\eta}-\Delta'\eta^{2}+C|\nabla\eta|^{2}-R(\dr)\eta)(M_{r}-\Delta u)\dr_{\beta}^{n}dt\\&\leq
Cr^{2n+2}+\frac{C}{r^{2}}\iint_{P'(r)}(M_{r}-\Delta u)\dr_{\beta}^{n}dt.
\end{align*}
Make use of weak Harnack inequality \eqref{eq:weak harnack} in Propositon \ref{prop-weak harnack}, we obtain that
\begin{align*}
 (\frac{r}{2})^{-2n}\iint_{P'(\frac{r}{2})}|\nabla\dr|^{2}\dr_{\beta}^{n}dt&\leq C\left(r^{-2n-2}\iint_{P'(r)}(M_{r}-\Delta u)\dr_{\beta}
^{n}dt+r^{2}\right)\\&\leq C(\inf_{P'(\frac{r}{2})}(M_{r}-\Delta u)+r^{2})\leq C(M_{r}-M_{\frac{r}{2}}+r^{2}).
\end{align*}
Hence, if \eqref{eq:small energy} doesn't hold for $r=\frac{\tilde{r}}{2},\frac{\tilde{r}}{2^{2}},\cdots,\frac{\tilde{r}}{2^{k}},$ then
$$k\epsilon_{0}\leq C(M_{\tilde{r}}-M_{\frac{\tilde{r}}{2^{k}}}+2\tilde{r}^{2}).$$
This is impossible for sufficiently large $k.$ Then the lemma follows.
\end{proof}

Finally we can complete the proof of Theorem \ref{mainthm} by an iteration argument as \cite{Ti14}. Suppose $\sigma=\lambda r$ where
$0<\lambda<1.$ Then suppose $r$ satisfies Lemma \ref{lem-small energy}, and denote $H(r)=\iint_{P(r)}|\nabla\dr|^{2}\dr_{\beta}^{n}dt,$
by Lemma \ref{lem-3.4} we obtain that $$\sigma^{-2n}H(\sigma)+\sigma^{2(\overline{\beta}^{-1}-1)}\leq c(\lambda^{2(\overline{\beta}^{-1}
-1)}+\frac{\epsilon_{0}}{\lambda^{2n}})r^{-2n}H(r)+(c\frac{r^{4-2\overline{\beta}^{-1}}}{\lambda^{2n}}+\lambda^{2\overline{\beta}^{-1}
-2})r^{2(\overline{\beta}^{-1}-1)}.$$ Now for $\alpha\in(0,\overline{\beta}^{-1}-1,1\}),$ choose $\lambda\in(0,1),$ such that
$$\log 2(1+c)\leq-2(\overline{\beta}^{-1}-1-\alpha)\lambda.$$ Then we choose $\theta=2(1+c)\lambda^{2(\overline{\beta}^{-1}-1)}<1.$ After
that we choose $\epsilon_{0}\leq\frac{\theta\lambda^{2n}}{2}$ and $r_{0}$ such that $$c\frac{r^{4-2\overline{\beta}^{-1}}}{\lambda^{2n}}+
\lambda^{2\overline{\beta}^{-1}-2}\leq\theta,$$ where we need to note that we can assume $\overline{\beta}\geq\frac{1}{2}$ otherwise the
high order estimate follows from Calabi's 3rd estimate directly (\cite{Br}). Then we can obtain that $$\sigma^{-2n}H(\sigma)+
\sigma^{2(\overline{\beta}^{-1}-1)}\leq\theta(r^{-2n}H(r)+r^{2(\overline{\beta}^{-1}-1)}).$$
Now by iteration method, for $r<r_{0}$ and any positive integer $k$ we have that
\begin{align*}
(\lambda^{k}r)^{-2n}H(\lambda^{k}r)+(\lambda^{k}r)^{2(\overline{\beta}^{-1}-1)}&\leq\theta^{k}(r^{-2n}H(r)+r^{2(\overline{\beta}^{-1}-1)})
\\&=(2(1+c)\lambda^{2(\overline{\beta}^{-1}-1-\alpha)})^{k}(\lambda^{k}r)^{2\alpha}r^{-2\alpha}(r^{-2n}H(r)+r^{2(\overline{\beta}^{-1}-1)}).
\end{align*}
By the choice of $\lambda$ above, $2(1+c)\lambda^{2(\overline{\beta}^{-1}-1-\alpha)}\leq 1,$ then the main theorem \ref{mainthm} follows.


\begin{thebibliography}{99}
\bibitem{Ber}{R. Berman, A thermodynamical formalism for Monge-Ampere equations, Moser-Trudinger
 inequalities and \ka-Einstein metrics, arXiv:1011.3976.}
\bibitem{Br}{S. Brendle, Ricci flat \ka\ metrics with edge singularities. International Mathematics Research
Notices, 2013(24): 5727-5766.}
\bibitem{CGP}{F. Campana, H. Guenancia, M, Paun. Metrics with cone singularities along normal crossing divisors
 and holomorphic tensor fields, Ann. Scient. Ec. Norm. Sup. 46 (2013), p. 879-916.}
\bibitem{CCT}{J. Cheeger, T. H. Colding and G. Tian, On the singularities of spaces with
bounded Ricci curvature, Geom. Funct. Anal., 12 (2002), 873-914.}
\bibitem{CDS1}{X.X. Chen, S. Donaldson and S. Sun, \ka-Einstein metrics on Fano manifolds, I:
approximation of metrics with cone singularities, arXiv:1211.4566.}
\bibitem{CDS2}{X.X. Chen, S. Donaldson and S. Sun, \ka-Einstein metrics on Fano manifolds, II:
limits with cone angle less than $2\pi$, arXiv:1212.4714.}
\bibitem{CDS3}{X.X. Chen, S. Donaldson and S. Sun, \ka-Einstein metrics on Fano manifolds, III:
limits as cone angle approaches $2\pi$ and completion of the main proof, arXiv:1302.0282.}
\bibitem{CW1}{X.X. Chen and Y.Q. Wang, Bessel functions, Heat kernel and the Conical \ka-Ricci flow. arXiv:1305.0255.}
\bibitem{CW2}{X.X. Chen and Y.Q. Wang, On the long time behaviour of the Conical \ka-Ricci flows.arXiv:1402.6689.}
\bibitem{Do}{S.K. Donaldson, \ka\ metrics with cone singularities along a divisor. arXiv:1102.1196.}
\bibitem{Ge}{W. Gehring, The $L^{p}$-integrability of the partial derivatives of a quasiconformal mapping,
Acta Math. 130 (1973), 265-277.}
\bibitem{GP}{H. Guenancia, M. Paun, Conic singularities metrics with perscribed Ricci curvature: the case
of general cone angles along normal crossing divisors. arXiv1307.6375.}
\bibitem{GT}{D. Gilbarg, N.S. Trudinger, Elliptic partial differential equations of second order, 2nd ed.,
Springer-Verlag (2001).}
\bibitem{Gu}{L.K. Gu, Parabolic partial differential equations of second order (Chinese Edition), Xiamen University Press, 1995.}
\bibitem{H}{R. Hamilton, Three-manifolds with positive Ricci curvature, J. Differential Geom. 17 (1982), no. 2, 255–306.}
\bibitem{HL}{Q. Han, F.H. Lin, Elliptic partial differential equations, Courant Inst. Math. Sci.; Amer. Math. Soc., 1997.}
\bibitem{Je}{T. Jeffres, Uniqueness of \ka-Einstein cone metrics, Publ. Math. 44 (2000).}
\bibitem{JMR}{T. D. Jeffres, R. Mazzeo, Y. A. Rubinstein, \ka-Einstein metrics with edge singularities,
 with an appendix by C. Li and Y. A. Rubinstein. arXiv:1105.5216.}
\bibitem{MRS}{R. Mazzeo, Y. Rubinstein, N. Sesum, Ricci flow on surfaces with conic singularities.arXiv:1306.6688.}
\bibitem{LS}{C. Li and S. Sun, Conic \ka-Einstein metrics revisited, arXiv:1207.5011.}
\bibitem{LSU}{O.A. Ladyzenskaja, V.A. Solonnikov, N.N. Uralceva,  Linear and quasi-linear equations of parabolic type.
Translations of Mathematical Monographs 23, Providence, RI: American Mathematical Society.}
\bibitem{LZ}{J. Liu and X. Zhang, The conical \ka-Ricci flow on Fano manifolds, arXiv:1402.1832v2.}
\bibitem{Sh}{L.M. Shen, Maximal time existence of unnormalized conical \ka-Ricci flow, arXiv:1411.7284.}
\bibitem{ST1}{J. Song and G. Tian, The \ka-Ricci flow on surfaces of positive Kodaira dimension,
 Invent. math., 170(2007), no. 3, 609-653.}
\bibitem{ST2}{J. Song and G. Tian, Canonical measures and \ka-Ricci flow. J. Amer. Math. Soc. 25(2012), 303-353.}
\bibitem{ST3}{J. Song and G. Tian, The \ka-Ricci flow through singularities, arXiv:0909.4898.}
\bibitem{Ti83}{On the existence of solutions of a class of Monge-Ampere equations, Acta Math. Sinica 4 (1988), 250–265.}
\bibitem{Ti98}{G. Tian, Canonical Metrics in \ka\ Geometry, Birkh¨auser, 2000.}
\bibitem{Ti12}{G.Tian, K-stabilitiy and \ka-Einstein metrics, arXiv:1211.4669.}
\bibitem{Ti14}{G.Tian, A 3rd derivative estimate for conic \ka\ metrics, preprint.}
\bibitem{TZ}{G. Tian and Z. Zhang, On the \ka-Ricci flow on projective manifolds of general type.
Chinese Annals of Mathematics - Series B, Volume 27, Number 2, 179–192.}
\bibitem{W}{Y.Q. Wang, Smooth approximations of the Conical \ka-Ricci flows. arXiv1401.5040.}
\bibitem{Yau}{S.T.Yau, On the Ricci curvature of a compact \ka\ manifold and the complex Monge-Ampere equation.
I. Comm. Pure Appl. Math. 31(1978), no. 3, 339-411.}
\end{thebibliography}
\end{document}